\documentclass[11pt,a4paper,english]{article}
\usepackage[english]{babel}

\usepackage[DIV=12]{typearea} 

\usepackage{blindtext}

\usepackage[dvipsnames]{xcolor}

\usepackage{amsthm, amsmath, amsfonts, amssymb, amscd, mathrsfs, comment}
\usepackage{ragged2e}
\usepackage{tikz-cd} 
\usepackage[symbol]{footmisc}

\newcommand{\gen}[1]{\langle #1\rangle}

\newcommand\Dir{\mathop{\rm Dr}}
\newcommand\ac{\operatorname{AC}}

\newcommand\DCC{\operatorname{DCC}}
\newcommand\DCCinf{\DCC{-\infty}}
\newcommand\RCC{\operatorname{RCC}}

\usepackage{tikz}
\usepackage{hyperref}

\newtheorem{theorem}{Theorem}
\newtheorem{theo}[theorem]{Theorem}
\newtheorem{lem}[theorem]{Lemma}

\newtheorem{cor}[theorem]{Corollary}
\newtheorem{prop}[theorem]{Proposition}

\newtheorem{rem}[theorem]{Remark}

\numberwithin{theorem}{section}
\date{}
\title{An antichain condition for infinite groups
\footnote{\,The authors are members of the non-profit association ‘‘AGTA --- Advances in Group Theory and Applications’’ (www.advgrouptheory.com), and are supported by GNSAGA (INdAM). Funded by the European Union - Next Generation EU, Missione 4 Componente 1 CUP B53D23009410006, PRIN 2022- 2022PSTWLB - Group Theory and Applications. The first author is supported by the FRA project FORMALG of the University of Naples Federico II (CUP E65F22000060001).
}
}
\author{Mattia Brescia, Bernardo Di Siena, Alessio Russo}

\begin{document}

\maketitle

\begin{abstract}
Let $\chi$ be a subgroup-theoretical property. We introduce an \emph{antichain condition} $\ac_\chi$ which forbids the existence of infinite antichains of mutually permutable non-$\chi$ subgroups whose infinite joins remain non-$\chi$. This is a ``width'' analogue of the real chain condition on non-$\chi$ subgroups, and it extends the usual hierarchy of weak chain conditions (double chain condition, deviation, and $\RCC$).

Our main results show that, within the universe of generalized radical groups, the antichain condition is as rigid as the corresponding chain conditions. For the properties $\chi$ of normality, almost normality, near normality, permutability, modularity, and pronormality, we prove that a generalized radical group satisfies $\ac_\chi$ if and only if it satisfies $\RCC$ on non-$\chi$ subgroups; equivalently, it satisfies any of the standard weak chain conditions on non-$\chi$ subgroups. In particular, we obtain minimax-type dichotomies: either the group is minimax, or \emph{every} subgroup satisfies $\chi$. This yields characterizations in terms of Dedekind groups, quasi-Hamiltonian groups, groups with modular subgroup lattice, and $\overline{T}$-groups. In the pronormal case, one has to deal with locally finite simple groups and a use of the Classification of Finite Simple Groups seems unavoidable.
\end{abstract}

\medskip

\noindent{\it Mathematics Subject Classification \textnormal(2020\textnormal)}: Primary 20E15, 20F22; Secondary 20F19, 20F24.
\medskip

\noindent{\it Keywords}: antichain condition; subgroup lattice; chain conditions; FC-groups; locally finite simple groups.

\section{Introduction}

Finiteness conditions on lattices of subgroups have long been a driving force in infinite group theory. Beyond the classical maximal and minimal conditions, several \emph{weak} chain conditions have proved particularly effective in forcing strong structural constraints on soluble and generalized soluble groups. Among the most studied are the \emph{double chain condition}, the \emph{weak double chain condition} $\DCCinf$, \emph{deviation}, and the \emph{real chain condition} $\RCC$.

Recall that a poset $\Lambda$ satisfies (weak) double chain condition if it contains no subposet with the order type of $\mathbb{Z}$ (where each subgroup has infinite index in its successor). It has \emph{deviation} if it contains no subposet with the order type of $\mathbb{Q}$ (see, for instance, \cite[6.1.13]{McCRob} and \cite{dGKR}), and it satisfies $\RCC$ if it contains no subposet with the order type of $\mathbb{R}$. These conditions can be imposed not only on the full subgroup lattice, but also on the family of subgroups failing a given property $\chi$ (or, dually, on the family of $\chi$-subgroups). This point of view goes under the broad tag of works on groups with ``few non-$\chi$ subgroups''.

A systematic investigation of weak chain conditions on restricted systems of subgroups has been recently carried out in a series of papers. Beyond the double chain condition on subnormal subgroups
and on non-subnormal subgroups,
several natural classes of $\chi$-subgroups were treated, including abelian and non-abelian subgroups \cite{BR1,BR2}, and non-pronormal subgroups \cite{BdG1}.
 These results highlight a recurring phenomenon in generalized radical groups: weak chain conditions typically collapse to minimax-type alternatives, often forcing the group to belong to an extremal class where every subgroup satisfies the relevant property. Related ``extremal'' subgroup constraints involving pronormality (for instance, metahamiltonian and abelian-or-pronormal groups) are studied in \cite{MetaHam,AbOrPro}.

More recently, Dardano and De Mari introduced and studied the real chain condition on non-$\chi$ subgroups, showing that it provides a flexible framework that simultaneously generalizes $\DCCinf$ and deviation, yet remains strong enough to yield sharp structure theorems for generalized radical groups \cite{Dardo,Dardo2}. The present paper complements this ``depth'' perspective by introducing a ``width'' analogue based on antichains, generalizing and sharpening most of their theorems.

\medskip

\noindent\textbf{An antichain condition.}
Let $\chi$ be a subgroup property. We say that a group $G$ satisfies $\ac_\chi$ if it does not contain an infinite antichain $(H_\lambda)_{\lambda\in\Lambda}$ of mutually permutable subgroups such that
\begin{enumerate}
\item for every $\lambda\in\Lambda$, $H_\lambda \not\le \bigvee_{\gamma\in\Lambda\setminus\{\lambda\}} H_\gamma$;
\item for every infinite subfamily $\{H_\lambda\mid \lambda\in\Lambda_0\}$, the join $\bigvee_{\lambda\in\Lambda_0} H_\lambda$ is \emph{not} a $\chi$-subgroup of $G$.
\end{enumerate}
By construction, $\ac_\chi$ is weaker than $\RCC$ on non-$\chi$ subgroups; in fact, we prove in Lemma~\ref{rob} that $\RCC$ implies $\ac_\chi$.

In this paper, we want to give the broadest setup in which all the aforementioned chain conditions live, showing how lattice ``width'' is indeed a new, more general way to look at lattice ``depth'' in infinite groups.

\medskip

\noindent\textbf{Scope and main results.}
We work throughout in the class of \emph{generalized radical groups}, i.e.\ groups possessing an ascending normal series whose factors are locally finite or locally nilpotent. This class is large enough to cover most of the classical classes of groups satisfying finiteness conditions, yet rigid enough for the standard pathological classes of groups (e.g. Tarski monsters) to be avoided.

Our first set of applications concerns the properties of being \emph{almost normal} and \emph{nearly normal}. In Section~3 we prove that, for generalized radical groups, $\ac_{an}$ and $\ac_{nn}$ are equivalent to $\DCCinf$ on non-(almost normal) and non-(nearly normal) subgroups, respectively, and hence to $\RCC$ and deviation on the same systems; equivalently, every non-minimax subgroup must be almost normal (resp.\ nearly normal). This is proved in Theorem~\ref{thannn} and recovers and sharpens the minimax alternative established for the corresponding chain conditions, and it connects with the study of groups having few subnormal non-normal subgroups
.

We then specialize to \emph{normality}. As an application of the result just stated, in the same section we obtain that, for generalized radical groups, $\ac_n$ is equivalent to $\DCCinf$ on non-normal subgroups and to the other weak chain conditions on non-normal subgroups; moreover, it forces a strong dichotomy: the group is either minimax or Dedekind (Theorem~\ref{fullnormal}).

Section~4 is devoted to the properties of being \emph{permutable} (quasinormal) and \emph{modular}. Here we show that $\ac_{per}$ and $\ac_{mod}$ (see, respectively, Theorem \ref{radhami} and \ref{thnm}) again coincide with $\DCCinf$ on the corresponding families of non-$\chi$ subgroups, yielding dichotomies ``minimax versus extremal'' (quasi-Hamiltonian groups in the permutable case, and groups with modular subgroup lattice in the modular case).

Finally, Section~5 treats \emph{pronormal} subgroups. Besides the same minimax-versus-extremal pattern (see the main Theorem \ref{thpr}), new phenomena appear in the locally finite case: controlling pronormality interacts with the structure theory of locally finite simple groups, in the spirit of \cite{BRcentral,LFSPron} and of the recent work on the pronorm of a group \cite{Pronorm}. The structural result about locally (radical-by-finite) groups was previously claimed in 
\cite{Dardo2}. While the statement is valid, the proof given there contain gaps in the passage to the infinite locally finite case. Our Proposition~\ref{simplepro} fills this gap using detailed information on families of finite simple groups; accordingly, the Classification of Finite Simple Groups enters the argument.

\medskip

The paper is organized as follows. Section~2 establishes general tools for extracting ``good'' subgroups from infinite direct products under $\ac_\chi$, and records the basic implications holding among the different chain conditions. Sections~3--5 apply these tools to the subgroup properties listed above and provide the corresponding structure theorems.

Most of our notation is standard and can be found in \cite{Rob72}.

\section{Preliminaries}

We begin with the following remark about the hierarchy of chain condition classes. Its straightforward proof can be found in \cite{Dardo}.

\begin{rem}
    Let $\Lambda$ be a poset of subgroups of a group $G$. The following hold:
    \begin{itemize}
        \item[1)] If $\Lambda$ satisfies $\DCCinf$, then $\Lambda$ has deviation.
        \item[2)] If $\Lambda$ has deviation, then $\Lambda$ satisfies $\RCC$.
        \item[3)] If $\Lambda$ satisfies $\RCC$ and is complete, then $\Lambda$ has deviation.
    \end{itemize}
\end{rem}

The $\RCC$ is the weakest of these conditions. However, the following lemma shows that the antichain conditions are a further generalization of the real chain conditions.

\begin{lem}\label{rob}
    Let $\chi$ be a subgroup property. Then the class of groups satisfying the $\RCC$ on non-$\chi$ subgroups is a subclass of the class of groups satisfying $\ac_\chi$.
\end{lem}

\begin{proof}
Let $(H_\lambda)_{\lambda\in\Lambda}$ be an infinite antichain of mutually permutable subgroups of $G$ such that:
    \begin{enumerate}
        \item For every $\lambda \in \Lambda$, $H_\lambda \not\le \bigvee_{\gamma\in\Lambda\setminus\{\lambda\}}H_\gamma$;
        \item The join of every infinite subfamily of $(H_\lambda)_{\lambda\in\Lambda}$ is not a $\chi$-subgroup of $G$.
    \end{enumerate}
Passing to a countably infinite subfamily, we may assume that $\Lambda$ is countable, and reindex the family as $(H_q)_{q\in\mathbb{Q}}$. For all $r \in \mathbb{R}$, define $K_r = \langle H_q \mid q < r, q \in \mathbb{Q} \rangle$. Due to the properties of $(H_q)_{q\in\mathbb{Q}}$, the family $(K_r)_{r\in\mathbb{R}}$ constitutes a real chain of non-$\chi$ subgroups of $G$.
\end{proof}

The following lemma allows to find ``good'' subgroups, under the assumption that $\chi$ satisfies some specific closure properties. We will say that $\chi$ is {\it normalizing join-closed} if, for any two $\chi$-subgroups $H$ and $K$ of $G$, if $H\leq N_G(K)$, then $HK$ is a $\chi$ subgroup. Finally, $\chi$ will be said to be {\it finite-intersection-closed} if, for any two $\chi$-subgroups $H$ and $K$ of $G$, then $H\cap K$ is a $\chi$ subgroup.

\begin{lem}\label{int1}
Let $\chi$ be a subgroup property. Let $G$ be a group satisfying $\ac_\chi$ and $L$ a subgroup of $G$. Suppose $G$ has a section $X/Y$ which is the direct product of an infinite collection $(X_\lambda /Y)_{\lambda \in \Lambda}$ of non-trivial subgroups such that $L \cap X \le Y$ and $LX_\lambda = X_\lambda L$ for every $\lambda\in\Lambda$. Then we have the following.
\begin{itemize}
    \item[(i)] If $\chi$ is normalizing join-closed, then $LX$ is a $\chi$-subgroup of $G$.
    \item[(ii)] There exists a normal subgroup $\widehat{X}$ of $X$ such that $L\widehat{X}$ is a $\chi$-subgroup of $G$. Moreover, if $\chi$ is normalizing join-closed, $\widehat{X}$ can be chosen to be a $\chi$-subgroup of $G$.
    \item[(iii)] If $\chi$ is finite-intersection-closed, then $LY$ is a $\chi$-subgroup of $G$.
\end{itemize}
\end{lem}
\begin{proof}
Let $\Omega$ and $\Gamma$ be two disjoint infinite subsets of $\Lambda$ such that $\Lambda = \Omega \cup \Gamma$.
\begin{itemize}
    \item[(i)] Let $H=\langle X_\omega\rangle_{\omega\in\Omega}$ and $K=\langle X_\gamma\rangle_{\gamma\in\Gamma}$. Consider the families $(LHX_\gamma)_{\gamma\in\Gamma}$ and $(LKX_{\omega})_{\omega\in\Omega}$. By the property $\ac_\chi$, there exist subsets $\Omega_0\subset\Omega$ and $\Gamma_0\subset\Gamma$ such that $K_0=LK\langle X_\omega\rangle_{\omega\in\Omega_0}$ and $H_0=LH\langle X_\gamma\rangle_{\gamma\in\Gamma_0}$ are $\chi$-subgroups of $G$. Thus, $\langle H_{0}, K_0 \rangle = LX$ is a $\chi$-subgroup of $G$.
    \item[(ii)] Consider the family $(LX_\lambda)_{\lambda\in\Lambda}$. We observe that for every $\lambda\in\Lambda$,
    \[ LX_\lambda \cap L\langle X_\alpha \rangle_{\alpha\in\Lambda\setminus\{\lambda\}} = LY. \]
    The result follows from the definition of $\ac_\chi$. Assume now that $\chi$ is normalizing join-closed and write $X/Y=\Dir_{\lambda\in\Lambda} W_\lambda/Y$, where each $W_\lambda/Y$ is a direct product countably many of $X_\lambda$. By point (i), each $W_\lambda$ is a $\chi$-group which is clearly still permutable with $L$. Then, repeating the argument before, using $W_\gamma$ instead of $X_\lambda$, we have the result.
    \item[(iii)] By point (ii), there exist the subgroups $\widehat{H}\leq\langle X_\omega\rangle_{\omega\in\Omega}$ and $\widehat{K}\leq\langle X_\gamma\rangle_{\gamma\in\Gamma}$ such that $L\widehat{H}$ and $L\widehat{K}$ are $\chi$-subgroups of $G$. It follows that $LY=L\widehat{H} \cap L\widehat{K}$, is a $\chi$-subgroup.
\end{itemize}
\end{proof}

\begin{lem}\label{inftn}
Let $\chi$ be a subgroup property. Let $G$ be a group satisfying $\ac_\chi$ and let $X/Y$ be a section of $G$ that is the direct product of an infinite collection $(X_\lambda /Y)_{\lambda \in \Lambda}$ of non-trivial subgroups. The following hold.
    \begin{itemize}
        \item[(i)] If $\chi$ is finite-intersection-closed, then $Y$ is a $\chi$-subgroup of $G$ and $X_\lambda$ is a $\chi$-subgroup of $G$ for each $\lambda\in\Lambda$.
        \item[(ii)] If every ascendant $\chi$-subgroup of $G$ is normal in $G$, and $X$ is ascendant in $G$, then $X_\lambda$ is normal in $G$ for every $\lambda\in\Lambda$. Moreover, also $Y$ is normal in $G$.
    \end{itemize}
\end{lem}
\begin{proof}
    \begin{itemize}
        \item[(i)] This is an immediate consequence of Lemma \ref{int1}(iii), choosing $L=Y$, or $X=\langle X_\alpha\rangle_{\alpha\in\Lambda\setminus\{\lambda\}}$ and $L=X_\lambda$.
        \item[(ii)] Let $\Omega$ and $\Gamma$ be two disjoint infinite subsets of $\Lambda\setminus\{\lambda\}$ such that $\Lambda\setminus\{\lambda\} = \Omega \cup \Gamma$. Using Lemma \ref{int1}(ii) with $L=1$, we find the subgroups $\widehat{H}\leq\langle X_\omega\rangle_{\omega\in\Omega}$ and $\widehat{K}\leq\langle X_\gamma\rangle_{\gamma\in\Gamma}$ such that $\widehat{H}$ and $\widehat{K}$ are $\chi$-subgroups of $G$. Since they are both ascendant in $G$, they are also normal in $G$ and so is their intersection $X_\lambda$. Finally, taking two different $\lambda,\mu$ in $\Lambda$, we have that $Y=X_\lambda\cap X_\mu$ is normal in $G$.
    \end{itemize}
\end{proof}


The following lemma is a straightforward consequence of Lemma \ref{int1}(i) and Lemma \ref{int1}(iii).

\begin{lem}\label{good}
    Let $\chi$ be a subgroup property which is normalizing join-closed and finite-inter\-section-closed. Let $G$ be a group and $A$ a subgroup of $G$ that is the direct product of an infinite family of non-trivial cyclic subgroups $(C_\lambda)_{\lambda\in \Lambda}$.
    If $G$ satisfies $\ac_\chi$, then every subgroup of $A$ is a $\chi$-subgroup. If in addition every $C_\lambda$ is permutable in $G$, then every cyclic subgroup of $G$ is a $\chi$-subgroup.
\end{lem}

We aim to apply these results to subgroup properties that generalize normality within the universe of radical groups and generalized radical groups.
Recall that a group $G$ is {\it radical} if it possesses a (normal) ascending series with locally nilpotent factors.
A group $G$ is {\it generalized radical} if it possesses a (normal) ascending series with factors that are locally finite or locally nilpotent.
In most cases, we can reduce the study of generalized radical groups with $\ac_\chi$ to the radical-by-finite case, as in \cite{Dardo}.

\begin{prop}\label{radbyfin}
    Let $\mathfrak{X}$ be a class of groups closed under taking subgroups and homomorphic images, such that any locally finite $\mathfrak{X}$-group is soluble-by-finite. Then any generalized radical $\mathfrak{X}$-group is a radical-by-finite $\mathfrak{X}$-group.
\end{prop}

\section{Almost normal and nearly normal subgroups}

We now apply the general strategy from the preliminaries to the posets of almost normal and nearly normal subgroups. Let $G$ be a group and $H$ a subgroup of $G$. The subgroup $H$ is said to be {\it almost normal} in $G$ if the normalizer $N_G(H)$ has finite index in $G$. We denote this by $H \operatorname{an} G$. On the other hand, the subgroup $H$ is said to be {\it nearly normal} in $G$ if it has finite index in its normal closure $H^G$. We denote this by $H \operatorname{nn} G$. One obvious, yet important fact is that if a cyclic subgroup of $G$ is either almost normal or nearly normal, then it is contained in the $FC$-centre of $G$.

It is easy to verify that the properties $an$ and $nn$ are finite-intersection-closed and normalizing join-closed, so we shall be able to apply the preceding results as needed.

Let $G$ be a group. The set of all $FC$-elements of $G$ is a subgroup and is usually denoted by $FC(G)$. If $G$ is abelian, the number $r=r_0(G)+\sum_{p\in \mathbb{P}}r_p(G)$ will be said to be the {\it total rank} of $G$, where $r_0$ and $r_p$ are the usual Prüfer torsion-free rank and $p$-rank of $G$, respectively.

\begin{prop}\label{fc}
    Let $G$ be a group satisfying $\ac_{an}$ (resp. $\ac_{nn}$). If $G$ is an $FC$-group, then $G$ is central-by-finite (resp. finite-by-abelian).
\end{prop}

\begin{proof}
Since $G'$ and $G/Z(G)$ are periodic, we can assume without loss of generality that $G$ is periodic. Let $A$ be an abelian subgroup of $G$. If $A$ is not \v{C}ernikov, then $A$ has a homomorphic image of infinite total rank. Thus, by Lemma \ref{int1}(i) with $L=1$, $A$ is almost normal (resp. nearly normal) in $G$. If on the other hand $A$ is a \v{C}ernikov group, then its finite residual is contained in $Z(G)$ and hence $A/A_G$ is finite, so by Dietzmann's Lemma, $A^G/A_G$ is finite, too, implying that $A$ is both almost normal and nearly normal in $G$. The result then follows from the well-known results of Eremin (see \cite{Er}) and Tomkinson (see \cite{Tom}).
\end{proof}

We begin by noting that, when the centre of the $FC$-centre has infinite total rank, the structure of a group satisfying $\ac_{an}$ (resp. $\ac_{nn}$) is very restricted.

\begin{theo}\label{best}
Let $G$ be a group satisfying $\ac_{an}$ (resp. $\ac_{nn}$). If $Z(FC(G))$ has infinite total rank, then $G$ is central-by-finite (resp. finite-by-abelian).
\end{theo}

\begin{proof}
Let $Z=Z(FC(G))$, let $T$ be the torsion subgroup of $Z$ and let $D=\Dir_{i\in \mathbb N}C_i$ be a direct product of cyclic subgroups of $Z$, which exists by hypothesis.

\textbf{Case 1: $G$ satisfies $\ac_{an}$}.

By Lemma \ref{good}, every subgroup of $D$ is almost normal in $G$. In particular, for every finite subgroup $F$ of $D$, $F^G$ is finitely generated and, writing $D=F^G\times E$, $E_G$ has still infinite total rank; then $D$ is the direct product of infinitely many $G$-invariant subgroups. Then every cyclic subgroup of $G$ is almost normal by Lemma \ref{int1}(iii) and hence $G$ is an $FC$-group. The conclusion now follows from Proposition \ref{fc}.

\textbf{Case 2: $G$ satisfies $\ac_{nn}$}.

By Lemma \ref{good}, every subgroup of $D$ is nearly normal in $G$. Let $T_D$ be the torsion subgroup of $D$. If $D/T_D$ has infinite rank, one easily gets that every subgroup of $D/T_D$ is normal in $G$, so $G/T_D$ is an $FC$-group by Lemma \ref{good}, and $G'T_D/T_D$ is finite by Proposition \ref{fc}. In particular, if $T_D$ is finite, we are done, so we may assume that $T_D$ is infinite. In this case, \cite[Lemma 2.7]{Cas} yields the existence of a finite $G$-invariant subgroup $N$ of $T$ such that every subgroup of $T/N$ is normal in $G$, so, as before, $G'$ is finite.
\end{proof}

The next theorem is the first global step: it shows that, in the presence of $\ac_{an}$ (or $\ac_{nn}$), radical-by-finite groups are forced into a minimax-type alternative. This will be found frequently in the rest of the paper.

\begin{theo}\label{minan}
Let $G$ be a radical-by-finite group satisfying $\ac_{an}$ (resp. $\ac_{nn}$). Then every non-minimax subgroup of $G$ is almost normal (resp. nearly normal) in $G$.
\end{theo}
\begin{proof}
Let $A$ be an abelian non-minimax subgroup of $G$, which we may find by \cite[Theorem 10.35]{Rob72}, and let $B$ be a free subgroup of $A$ such that $A/B$ is periodic. If $B$ has infinite total rank, then by Lemma \ref{good} it is contained in $FC(G)$. Since $FC$-groups are central-by-periodic, $Z(FC(G))$ has infinite total rank and the thesis follows from Theorem \ref{best}. Thus, we can suppose that $B$ is finitely generated, so that $A/B$ has an infinite socle $S/B$. Now, Lemma \ref{inftn}(i) yields that $B$ is almost normal (resp. nearly normal). If $B$ is nearly normal, $|B^G:B|=n$, say, we have that $(B^G)^{n!}\leq B$, so $B^G/B_G$ has finite exponent and hence it is finite. Therefore, in any case $B$ is almost normal. In particular, $C=B^G$ is finitely generated and $N_G(B)/B$ is finite-by-abelian by Theorem \ref{best}. Then $(C \cap N_G(B))/B$ is polycyclic-by-finite and so is $C$. However, since $C$ is finitely generated, $SC/C$ is still infinite and hence by Theorem \ref{best} we have that $G'C/C$ is finite and so $G'$ is polycyclic-by-finite. In particular, every non-minimax subgroup $H$ of $G$ has polycyclic-by-finite commutator subgroup and hence $H/H'$ contains a homomorphic image of infinite total rank. Therefore, $H$ is almost normal (resp. nearly normal) in $G$ by Lemma \ref{int1}(i) with $L=1$.
\end{proof}

\medskip

We now prove the main theorem of this section.

\begin{theo}\label{thannn}
    Let $G$ be a generalized radical group. The following conditions are equivalent:
    \begin{itemize}
        \item[(i)] $G$ satisfies Min-$\infty$ on non-almost normal (resp. non-nearly normal) subgroups;
        \item[(ii)] $G$ satisfies Max-$\infty$ on non-almost normal (resp. non-nearly normal) subgroups;
        \item[(iii)] $G$ satisfies $\DCCinf$ on non-almost normal (resp. non-nearly normal) subgroups;
        \item[(iv)] $G$ has almost normal (resp. nearly normal) deviation;
        \item[(v)] $G$ satisfies RCC on non-almost normal (resp. non-nearly normal) subgroups;
        \item[(vi)] $G$ satisfies $\ac_{an}$ (resp. $\ac_{nn}$);
        \item[(vii)] Every non-minimax subgroup of $G$ is almost normal (resp. nearly normal).
    \end{itemize}
\end{theo}
\begin{proof}
By \cite[Proposition 2.1, Theorem 3.7 and Theorem 3.10]{Dardo}, the first 5 conditions are equivalent. By Lemma \ref{rob}, they all imply (vi). Let us suppose (vi) holds. First, if $G$ is locally finite and non-\v Cernikov, it is finite-by-abelian by Proposition \ref{fc}. This and Proposition \ref{radbyfin} show that $G$ is in general radical-by-finite. By Theorem \ref{minan}, every non-minimax subgroup of $G$ is almost normal (resp. nearly normal) and in this case, $G$ obviously satisfies (i) and the proof is complete.
\end{proof}

Notice that the structure of groups satisfying the weak minimal condition on non-(nearly normal) subgroups and the weak minimal condition on non-(almost normal) subgroups has been described in detail in \cite{CuKu} and in \cite{DM}, respectively. So our description is complete.

\smallskip
We now obtain the corresponding result for normal subgroups.

\begin{theo}\label{fullnormal}
    Let $G$ be a generalized radical group. The following conditions are equivalent:
    \begin{itemize}
        \item[(i)] $G$ satisfies Min-$\infty$ on non-normal subgroups;
        \item[(ii)] $G$ satisfies Max-$\infty$ on non-normal subgroups;
        \item[(iii)] $G$ satisfies $\DCCinf$ on non-normal subgroups;
        \item[(iv)] $G$ has normal deviation;
        \item[(v)] $G$ satisfies RCC on non-normal subgroups;
        \item[(vi)] $G$ is minimax or Dedekind;
        \item[(vii)] $G$ satisfies $\ac_{n}$.
    \end{itemize}
\end{theo}

\begin{proof}
By \cite[Proposition 2.1 and Theorem 3.11]{Dardo}, the first 6 conditions are equivalent, and they all imply (vii) by Lemma \ref{rob}. Let us suppose (vii) holds. Certainly, $\ac_{n}$ implies $\ac_{nn}$, so $G$ is radical-by-finite by the previous theorem. Suppose that $G$ is not minimax, so that we may find an abelian non-minimax subgroup $A$ of $G$. Let $B$ be a torsion-free subgroup of $A$ such that $A/B$ is periodic. If $B$ has infinite rank, by Lemma \ref{inftn}(i) and Lemma \ref{good} all cyclic subgroups of $G$ are normal, so $G$ is a Dedekind group, so we may assume that $B$ is finitely generated. However, this means that for every integer $k>2$, $A/B^{2^k}$ has an infinite socle. Again by Lemma \ref{inftn}(i) and Lemma \ref{good}, $G/B^{2^k}$ is a Dedekind group for every $k>2$ and hence $G$ is abelian.
\end{proof}

\section{Permutable and modular subgroups}

We next consider the antichain conditions restricted to permutable and modular subgroups. Recall that $\ac_{per}$ is the antichain condition $\ac_\chi$ when $\chi$ is the property of being permutable. Although these classes are not, in general, subgroup-closed, the relevant posets still interact well with subgroups in generalized radical groups, so that one can push the minimax alternative obtained earlier to these settings.

Let $G$ be a group and $K$ a subgroup of $G$. The subgroup $K$ is said to be {\it permutable} (or {\it quasinormal}) in $G$ (denoted $K \operatorname{per} G$) if for every subgroup $H$ of $G$, $HK = KH$. Recall that a group is called {\it quasi-Hamiltonian} if all its subgroups are permutable.

\begin{prop}\label{minmaxperm}
    Let $G$ be a group satisfying $\ac_{per}$. If $G$ is radical-by-finite, then every non-minimax subgroup of $G$ is permutable in $G$.
\end{prop}
\begin{proof}
Let $H$ be a non-minimax subgroup of $G$. Since $H$ is radical-by-finite, $H$ contains an abelian non-minimax subgroup $A$ by \cite[Theorem 10.35]{Rob72}. Then $A$ has a homomorphic image $A/B$ that is the direct product of an infinite collection of non-trivial subgroups. Write now $A/B=\Dir_{i\in I} A_i/B$, where each $A_i/B$ is infinite. Then by Lemma \ref{inftn}, every $A_i$ is permutable in $G$. Let $x$ be an element of $H \setminus B$. We can select an infinite sub-collection of $(A_i/B)_{i\in I}$ generating a subgroup $A_x/B$ such that $\langle x \rangle \cap A_x \le B$. By Lemma \ref{int1}(ii), there exists a subgroup $\widehat{A}_x$ of $A_x$ containing $B$ which is permutable in $G$, such that $\langle x \rangle \widehat{A}_{x}$ is permutable in $G$. Since $H$ is generated by the join of such subgroups, $H$ is permutable in $G$.
\end{proof}

The following characterization sharpens the previous equivalence by isolating the exact role played by the Hirsch--Plotkin radical and by minimax sections. It will be the key input when we turn to permutable and modular subgroups.

\begin{theo}\label{radhami}
    Let $G$ be a generalized radical group. The following conditions are equivalent:
    \begin{itemize}
        \item[(i)] $G$ satisfies Min-$\infty$ on non-permutable subgroups;
        \item[(ii)] $G$ satisfies Max-$\infty$ on non-permutable subgroups;
        \item[(iii)] $G$ satisfies $\DCCinf$ on non-permutable subgroups;
        \item[(iv)] $G$ has permutable deviation;
        \item[(v)] $G$ satisfies RCC on non-permutable subgroups;
        \item[(vi)] $G$ is minimax or quasi-Hamiltonian;
        \item[(vii)] Every non-minimax subgroup of $G$ is permutable in $G$;
        \item[(viii)] $G$ satisfies $\ac_{per}$.
    \end{itemize}
\end{theo}
\begin{proof}
By \cite[Proposition 2.1 and Theorem 4.6]{Dardo}, the first 6 conditions are equivalent and clearly (i) and (vii) are equivalent. By Lemma \ref{rob}, they all imply (viii). Assume (viii) holds and suppose first that $G$ is locally finite and not \v Cernikov, so that we may find a subgroup $S$ which is the direct product of infinitely many non-trivial cyclic subgroups. Write $S=\Dir_{i\in \Lambda} S_i$, where each $S_i$ is infinite. Then by Lemma \ref{inftn}, every $S_i$ is permutable in $G$. Let now $x$ and $y$ be two elements of $G$ and assume without loss of generality that $\gen{x,y}\cap S=1$. Lemma \ref{int1}(ii) makes it possible to find a subgroup $\widehat{S}=\langle S_\gamma\rangle_{\gamma \in\Gamma}$ of $S$ for some $\Gamma\subset\Lambda$ such that $\gen{x}\widehat{S}$ is permutable in $G$. Clearly, since every $S_i$ is permutable in $G$, also $\widehat{S}$ is such. Now it easily follows that $\gen{x}\gen{y}\subseteq (\widehat{S}\gen{x})\gen{y}=\gen{y}\gen{x}\widehat{S}$. Since also the converse holds, similarly, we have that $\gen{x}\gen{y}=\gen{y}\gen{x}$ and then $G$ is quasi-Hamiltonian. It follows that $G$ is in any case soluble-by-finite (see \cite[Theorem 2.4.11]{Schmidt}) whenever $G$ is locally finite, and this shows, together with Proposition \ref{radbyfin}, that $G$ is in general radical-by-finite. Then we are in a position to apply Proposition \ref{minmaxperm} to obtain the result.
\end{proof}

\smallskip

We now move to modular subgroups. Let $G$ be a group and $K$ a subgroup of $G$. The subgroup $K$ is said to be {\it modular} in $G$ (denoted $K \operatorname{mod} G$) if for all subgroups $H, L$ of $G$:
\begin{itemize}
    \item If $H \le L$, then $L \cap \langle H, K \rangle = \langle H, K \cap L \rangle$;
    \item If $K \le H$, then $L \cap \langle H, K \rangle = \langle K, L \cap H \rangle$.
\end{itemize}

Recall that the property of being modular is join-closed (see for instance \cite[Theorem 2.1.6(d)]{Schmidt}). In the following, we will use Theorem 2.4.11 and Theorem 2.4.14 in \cite{Schmidt} to refer to the classification of quasi-Hamiltonian groups made by Iwasawa. Recall also that a subgroup is permutable if and only if it is modular and ascendant.

\begin{theo}\label{thnm}
    Let $G$ be a generalized radical group. The following conditions are equivalent:
    \begin{itemize}
        \item[(i)] $G$ satisfies Min-$\infty$ on non-modular subgroups;
        \item[(ii)] $G$ satisfies Max-$\infty$ on non-modular subgroups;
        \item[(iii)] $G$ satisfies $\DCCinf$ on non-modular subgroups;
        \item[(iv)] $G$ has modular deviation;
        \item[(v)] $G$ satisfies RCC on non-modular subgroups;
        \item[(vi)] Every non-minimax subgroup of $G$ is modular in $G$;
        \item[(vii)] $G$ is minimax or has a modular subgroup lattice;
        \item[(viii)] $G$ satisfies $\ac_{mod}$.
    \end{itemize}
\end{theo}
\begin{proof}
By \cite{Dardo} and \cite{DeMar}, the first 7 conditions are equivalent. By Lemma \ref{rob}, they all imply (viii). Suppose (viii) holds and that $G$ is not minimax. Assume first that $G$ is locally finite, let $X$ be a finite subgroup of $G$ and let $x$ be an element of $X$. A result of Zai\v cev (see \cite{Zai}) ensures that we may find an abelian $X$-invariant subgroup $E$ of $G$ which is the direct product of countably many cyclic groups. Clearly, we assume without loss of generality that $E\cap X=1$ Since $X$ is finite, we may write $E=\Dir_{i\in\mathbb{N}} E_i$, where each $E_i$ is $X$-invariant. Then we may apply Lemma \ref{int1}(ii) to find an $X$-invariant subgroup $\widehat{E}$ of $E$ which is modular in $G$ and such that $\widehat{E}\langle x\rangle$ is modular in $G$. Then $\widehat{E}\langle x\rangle\cap X=\langle\widehat{E}\cap X,x\rangle=\langle x\rangle$. This shows that every cyclic subgroup of $X$ is modular in $X$ and hence $\mathcal{L}(X)$ is modular. By the generality of $X$, it follows that $\mathcal{L}(G)$ is modular, too, and this means that $G$ is metabelian by the already quoted result by Iwasawa. An application of Proposition \ref{radbyfin} yields that in general $G$ is radical-by-finite.

Let now $H$ be a non-minimax subgroup of $G$, and let $A$ be a non-minimax subgroup of $H$. Then we may find a subgroup $B$ of $A$ such that $A/B$ is the direct product of infinitely many non-trivial subgroups. Let now $h$ be an element of $H\setminus B$. Clearly, $B$ can be taken in a way that $A/B$ has still infinite total rank and that $\langle h\rangle\cap A\leq B$. Now, like before, we may apply Lemma \ref{int1}(ii) to find an $\langle h\rangle$-invariant subgroup $\widehat{A}_h$ of $A$ which is modular in $G$ and such that $\widehat{A}_h\langle h\rangle$ is modular in $G$. From this we have that $H=\langle\langle \widehat{A}_h\langle h\rangle\rangle\mid h\in H\setminus B\rangle$ is modular in $G$ and we have that $\mathcal{L}(G)$ is modular by \cite{DM}.
\end{proof}

\section{Pronormal subgroups}

Finally we deal with pronormal subgroups. Here the poset is typically more rigid, and the antichain condition forces strong restrictions on the Hirsch--Plotkin radical. One main obstacle will be in fact dealing with groups with trivial Hirsch-Plotkin radical, but in that case it will be shown that one can always produce forbidden antichains, provided that the group is locally finite and not soluble-by-finite.

Let $G$ be a group and $H$ a subgroup of $G$. The subgroup $H$ is said to be {\it pronormal} in $G$ if, for every element $g \in G$, the subgroups $H$ and $H^g$ are conjugate in their join $\langle H, H^g \rangle$. Recall here that a group $G$ is said to be a $T$-group if it has no subnormal non-normal subgroups; moreover, $G$ is called a $\overline{T}$-group if every subgroup of $G$ is a $T$-group. Recall also that a locally (soluble-by-finite) group $G$ is a $\overline{T}$-group if and only if every cyclic subgroup of $G$ is pronormal (see for instance \cite[Corollary 2.1]{BdG1}).

Here $\ac_{pr}$ denotes the antichain condition $\ac_\chi$ when $\chi$ is the property of being pronormal. We first show that $\ac_{pr}$ rules out infinite total rank in a natural family of sections of the Hirsch--Plotkin radical. This is a crucial step towards the final minimax alternative for radical-by-finite groups under $\ac_{pr}$.

\begin{theo}\label{inftytotalrank}
    Let $G$ be a group satisfying $\ac_{pr}$. If the Hirsch-Plotkin radical $R$ of $G$ has infinite total rank, then $G$ is a $\overline{T}$-group.
\end{theo}

\begin{proof}
Since $R$ is locally nilpotent, it satisfies $\ac_n$. Thus, by Theorem \ref{fullnormal}, $R$ is a Dedekind group. So, $R$ contains a direct product $D$ of infinitely many non-trivial cyclic subgroups and hence by Lemma \ref{inftn}(ii) every cyclic subgroup of $R$ is normal in $G$. In particular, every subgroup of $R$ is normal in $G$. Let $H$ and $K$ be two subgroups of $G$ such that $H$ is subnormal in $K$ and assume first that $H \cap D$ is finitely generated. Then we can assume that $D\cap H=1$. Write now $D=D_1\times D_2$, where $D_1$ and $D_2$ are both of infinite total rank. Then by Lemma \ref{int1}(ii) there exist subgroups $\widehat{D}_1\le D_1$ and $\widehat{D}_2 \le D_2$ which are pronormal in $G$. This implies that $\widehat{D}_1H$ and $\widehat{D}_2H$ are pronormal in $\widehat{D}_1K$ and $\widehat{D}_2K$, respectively. However, since they are also subnormal in those groups, they are even normal. From this it follows that $K$ normalizes both $\widehat{D}_1H$ and $\widehat{D}_2H$ and hence also $H=(\widehat{D}_1H) \cap (\widehat{D}_2 H)$ is normal in $K$.

Assume now that $H \cap D$ is infinitely generated and let $x$ be an element of $G \setminus R$. By Lemma \ref{int1}(ii), there exists a subgroup $D_x$ of $H \cap D$ such that $D_x\langle x \rangle$ is pronormal in $G$. It follows that in the quotient $G/(H \cap D)$, all cyclic subgroups are pronormal. Thus, $G/(H \cap D)$ is a $\overline{T}$-group (see \cite[Corollary 2.1]{BdG1}). Since $H/(H \cap D)$ is subnormal in $K/(H \cap D)$, it is normal, and also in this case $H$ is normal in $K$.
\end{proof}

\begin{lem}\label{radT}
Let $G$ be a locally radical $\overline{T}$-group satisfying $\ac_{pr}$. Then every subgroup of $G$ is pronormal in $G$.
\end{lem}

\begin{proof}
Let $H$ be a finitely generated subgroup of $G$. Since a locally nilpotent $\overline{T}$-group is a Dedekind group, any section of $H$ contains a normal abelian subgroup, so $H$ is hyperabelian. Then $G$ is locally soluble and hence even metabelian, since soluble $\overline{T}$-groups are metabelian (see \cite{Robtra}). If $G$ is non-periodic, then $G$ is abelian, so we can assume that $G$ is periodic. Let $L = [G', G]$. Then $\pi(L) \cap \pi(G/L) = \emptyset$ and $2 \notin \pi(L)$. Let $\pi = \pi(G/L)$, and let $P$ be a Sylow $\pi$-subgroup of $G$. Note that $PL$ is a normal subgroup of $G$ and $P \cong PL/L$. Since $P$ is nilpotent, it satisfies $\ac_{n}$. If $P$ is not a \v{C}ernikov group, it is a Dedekind group, so, by Lemma \ref{int1}(i), $P$ is pronormal in $G$ and hence $G = N_G(P)L$, which is in contradiction to the fact that obviously $N_G(P)/P$ is a $\pi'$-group. Therefore, $P$ is a \v{C}ernikov group. Let now $F/PL$ be a non-trivial finite subgroup of $G/PL$. Since every subgroup of $L$ is normal in $G$, $G$ acts on $L$ as a group of power automorphisms, so, in particular, $G/C_G(L)$ is residually finite. Then, the finite residual $D$ of $P$ must be contained in $C_G(L)$. Now $F/C_F(L)$ is finite, so there exists a complement $X$ of $L$ in $F$ (see for instance \cite[Theorem 2.4.5]{Dix}). We have that $F$ is abelian-by-finite so its Sylow $\pi$-subgroups are conjugate but this is not possible for the choice of $F$. Therefore, we have that $G = PL$ for every Sylow $\pi$-subgroup $P$ of $G$ and so by \cite[Lemma 8]{Sub} we have that all subgroups of $G$ are pronormal in $G$.
\end{proof}

We can now combine the previous structural restrictions with the radical-by-finite hypothesis to obtain a clean dichotomy. This mirrors the pattern already met for the other classes of subgroups, and completes the parallel between the different antichain conditions studied in the paper.

\begin{theo}\label{radual}
    Let $G$ be a radical-by-finite group satisfying $\ac_{pr}$. Then $G$ is minimax or all its subgroups are pronormal in $G$.
\end{theo}

\begin{proof}
By Lemma \ref{radT}, it suffices to prove that if $G$ is not minimax, then $G$ is a $\overline{T}$-group. Assume $G$ is not minimax, so that the Hirsch-Plotkin radical $R$ of $G$ is not minimax, too. Since in a locally nilpotent group pronormality and normality coincide, $R$ is a Dedekind group by Theorem \ref{fullnormal}. If $R$ has infinite total rank, then by Theorem \ref{inftytotalrank}, $G$ is a $\overline{T}$-group, so we can suppose that $R$ has finite total rank. In particular, $R$ is abelian, $R=T\times A$, where $T$ is the torsion subgroup of $R$, which is a \v Cernikov group, and $A$ is not minimax.

Let $x$ be an element of infinite order in $R$ and let $F$ be a free abelian subgroup of $R$ containing $x$ such that $R/F$ is periodic. Let $L$ be a subgroup of finite index in $F$ containing $x$. Then $R/L$ has infinite total rank, so $L$ is normal in $G$ by Lemma \ref{inftn}(ii). Since $R$ has finite total rank, $F$ is finitely generated, so $\langle x\rangle$ equals the intersection of all the subgroups of finite index of $F$ containing it. In particular, $\langle x \rangle$ is normal in $G$. Since $R$ is generated by its aperiodic elements, it follows that every subgroup of $R$ is normal in $G$.

Suppose now for a contradiction that $G$ is not abelian. Then $C_G(R) = R$ has index 2 in $G$, and every element outside $R$ acts as the inversion on $R$. Write $G = \langle g \rangle R$ where $g^2$ belongs to $R$. Let $V$ be a free abelian subgroup of $R$ such that $R/V$ is periodic, and let $W = V^4$. Then $R/W$ has infinite total rank, so $G/W$ is a $\overline{T}$-group again by Theorem \ref{inftytotalrank}. However, $\langle g\rangle W$ is clearly subnormal and not normal in $G/W$ and this gives a contradiction.
\end{proof}

\begin{cor}\label{fgnmin}
    Let $G$ be a locally radical group satisfying $\ac_{pr}$. If $G$ has a finitely generated subgroup that is not minimax, then all subgroups of $G$ are pronormal in $G$.
\end{cor}

\begin{proof}
    Let $X$ be a finitely generated subgroup of $G$ that is not minimax. By Theorem \ref{radual}, $X$ is a $\overline{T}$-group. Since being a $\overline{T}$-group is a local property, every finitely generated subgroup of $G$ is a $\overline{T}$-group, implying $G$ is a $\overline{T}$-group.
    By Lemma \ref{radT}, every subgroup of $G$ is pronormal.
\end{proof}

\begin{cor}\label{LRCor}
    Let $G$ be a locally radical group of finite total rank satisfying $\ac_{pr}$. Then $G$ is minimax or all its subgroups are pronormal in $G$.
\end{cor}

\begin{proof}
    If there exists a finitely generated subgroup of $G$ that is not minimax, the result follows from Corollary \ref{fgnmin}.
    Suppose all finitely generated subgroups are minimax. Since $G$ has finite total rank, by Lemma 10.39 of \cite{Rob72}, $G$ is hyperabelian (and thus radical). The result follows from Theorem \ref{radual}.
\end{proof}

We now rule out the infinite locally finite simple case by a close inspection on families of finite simple groups.

\begin{prop}\label{simplepro}
Let $G$ be a locally finite group satisfying $\ac_{pr}$. Then $G$ is soluble-by-finite.
\end{prop}
\begin{proof}
Suppose first that $G$ is simple. Assume for a contradiction that $G$ is infinite. Suppose first that $G$ is linear, and let $B=T\ltimes P$ be a Borel subgroup of $G$, where $T$ is a maximal torus and $P$ is a maximal $p$-subgroup of $G$. It follows from \cite[Theorem A]{HartShu}, for example, that $G$ is a group of Lie type over an infinite locally finite field $K$ of positive characteristic, say $p$. Notice that a maximal $p$-subgroup of $G$ contains an elementary abelian subgroup of infinite total rank, so it must be a Dedekind group by Lemma \ref{inftn}(ii). Consider first the case in which $G$ is of type $A_1$ (i.e. $\operatorname{PSL}_{2}(K)$). Then $P$ is an infinite elementary abelian $p$-group. Since any $\operatorname{A}_1(p^n)$ has a semisimple element of order $\frac{p^n+1}{p-1}$, hence acting irreducibly on a unipotent subgroup of $\operatorname{A}_1(p^n)$, $P$ contains subgroups which are not pronormal in $B$ and this contradicts Lemma \ref{inftn}(ii).

Let now $G$ be a group of one of the types listed below, subject to the indicated restrictions on the characteristic $p$:

\begin{center}
\begin{tabular}{lll}
\textbf{Type of} $\mathbf{G}$ & \textbf{Rank} & \textbf{Restriction on} $\mathbf{p}$ \\
$B_n$ & $n > 1$ & $p \neq 2$ \\
$C_n$ & $n > 2$ & $p \neq 2$ \\
$G_2$ & --- & $p \notin \{2, 3\}$ \\
$E_6, E_7, E_8$ & --- & None \\
\end{tabular}
\end{center}

By \cite[Lemma 4.2]{HartShu}, the nilpotency class of a maximal $p$-subgroup of $G$ is the height $h(R)$ of the (unique) highest root $R$ as in \cite[Theorem 5.3.3]{Cart}. By \cite[Proposition 10.2.5 and the subsequent remark]{Cart}, we see that $h(R)\geq3$ for all the types in the table, and this contradicts the fact that $P$ is a Dedekind group.

For $G_2$ in characteristic $2$ or $3$, the same contradiction follows from \cite[Lemma 4.5]{HartShu}.

On the other hand, if $G$ is of type $B_n$ or $C_n$ over a field of characteristic $2$, \cite[Lemma 4.3]{HartShu} gives that the nilpotency class of $P$ is $2n-2$, so the only option is $n=2$ and hence $G$ is of type $B_2$;
however, also in this case one gets to a contradiction because $B_2(2)$ contains a dihedral subgroup of order $8$. 

\medskip

Suppose now that $G$ is a twisted group of the following types $$^2A_n\,\,(n> 1),\, ^2B_2,\, ^2D_n\,\, (n>3),\, ^3D_4,\, ^2E_6,\, ^2F_4,\, ^2G_2$$ and take first into account the cases in which $G$ is of type ${}^2 A_n$, ${}^2 D_n$ or ${}^2 E_6$. By \cite[Lemma 4.7]{HartShu}, the nilpotency class of $P$ is the same as that of the corresponding maximal $p$-subgroup in the untwisted case, which is, as already said, the height $h(R)$ of the highest root $R$. Again by \cite[Proposition 10.2.5]{Cart} we have that for ${}^2 A_n$, $h(R)$ is equal to $n$, for ${}^2 D_n$, $h(R)$ is equal to $2n-3$ and for ${}^2 E_6$, $h(R)$ is equal to $11$. By \cite[Lemmas 4.8, 4.9 and 4.10]{HartShu}, the nilpotency class of a Sylow $p$-subgroup of a group of type ${}^3D_4$, ${}^2 F_4$, or ${}^2G_2$ is $5$, $8$, or $4$, respectively. Thus, the only types remaining for $G$ are
$$^2A_2,\, ^2B_2.$$

Now, if $G$ is a group of type ${}^2B_2$ (namely a Suzuki group), then its Sylow $2$-subgroups are well-known to be not Dedekind (see for instance \cite[Theorem 6]{Suz}), so this case is ruled out.

Let finally $G$ be a group of type ${}^2A_2$, namely a unitary group $PSU(n+1,q)$. We already noticed that the nilpotency class of $P$ is $2$, so $p$ must be $2$. Take 
$$g = \left( \begin{array}{ccc} 
1 & 1 & b \\ 
0 & 1 & 1 \\ 
0 & 0 & 1 
\end{array} \right),\quad h=\left( \begin{array}{ccc} 
1 & s & t \\ 
0 & 1 & s^q \\ 
0 & 0 & 1 
\end{array} \right),
$$ subject to the following conditions: $b+b^q=1$, $t+t^q=s^{q+1}$ and $s+s^q\not \in\{0,1\}$. Notice that we can satisfy all these conditions because the trace map $z \in GF(q^2)\mapsto z+z^q\in GF(q)$ is surjective. Then it is easy to check that $g$ has period $4$ and that $h$ does not normalize $\langle g\rangle$, so that $P$ is not a Dedekind group.

This concludes our proof for the linear case.

\medskip

Assume now that $G$ is non-linear and let $\mathcal{K}=\{(H_i,M_i)\mid i\in I\}$ be a Kegel cover of $G$ (see \cite[Theorem 4.4 and Lemma 4.5]{KeWe}). By the already quoted \cite[Theorem A]{HartShu}, the family of finite simple groups $\{H_i/M_i\}$ has unbounded total rank. By the first part of the proof, growing rank always gives non-Dedekind $p$-subgroups for some prime $p$. However, \cite[Theorem B]{HartShu} yields that $G$ does not satisfy min-$p$, contradicting Lemma \ref{inftn}(ii).

This proves the statement in the simple case.

\medskip

Suppose now that $G$ is an infinite locally finite group and let $L$ be the locally soluble radical of $G$. We claim that $L$ is soluble. Assuming that $L$ does not satisfy the minimal condition, a theorem of \v Cernikov (see, for instance, \cite[Theorem 3.45]{Rob72}) yields that $L$ does not satisfy the minimal condition on abelian subgroups. Let $E$ be a finite subgroup of $L$. Then, by a result of Zaicev (\cite{Zai}) we may find an $E$-invariant abelian subgroup $A$ of $L$ not satisfying the minimal condition. Exchanging $A$ with its socle, we may assume that $A$ is the direct product of infinitely many non-trivial cyclic subgroups. Since $E$ is finite, we may take $A\cap E=1$ and write $A$ as the direct product of infinitely many $E$-invariant subgroups. Let now $F$ be a subnormal subgroup of $E$. An application of Lemma \ref{int1} (ii) makes it possible to find an $E$-invariant subgroup $\widehat{A}$ of $A$ such that $F\widehat{A}$ is pronormal in $G$. In particular, $F=(F\widehat{A})\cap E$ is pronormal in $E$ and hence normal in it. This shows that $E$ is a $T$-group, so that $E$ is metabelian and hence $L$ itself is metabelian.

\medskip

Finally, assume for a contradiction that $G/L$ is infinite and let $S$ be any Sylow $2$-subgroup of $G$, which does not satisfy the minimal condition by the first part of the proof and \cite[Lemma 3.32]{KeWe}. Then, every subgroup of $S$ is pronormal in $G$ by Lemma \ref{inftn} (ii), and \cite[Corollary 3.52]{BIT} shows that $G$ is locally soluble, which leads to the contradiction $G=L$.

\end{proof}

\begin{rem}
The statement was already claimed in \cite[Lemma 12]{Dardo2}. The proof given there contains a gap in the argument in the last paragraph, where in principle the Hirsch-Plotkin radical might be trivial (e.g. in an infinite simple locally finite group). The result is nevertheless correct, as shown in the proof provided above.
\end{rem}

We are now in a position to prove the main theorem of this section.

\begin{theo}\label{thpr}
    Let $G$ be a generalized radical group. The following conditions are equivalent:
    \begin{itemize}
        \item[(i)] $G$ satisfies Min-$\infty$ on non-pronormal subgroups;
        \item[(ii)] $G$ satisfies Max-$\infty$ on non-pronormal subgroups;
        \item[(iii)] $G$ satisfies $\DCCinf$ on non-pronormal subgroups;
        \item[(iv)] $G$ has pronormal deviation;
        \item[(v)] $G$ satisfies RCC on non-pronormal subgroups;
        \item[(vi)] Every non-minimax subgroup of $G$ is pronormal in $G$;
        \item[(vii)] $G$ is minimax or every subgroup of $G$ is pronormal;
        \item[(viii)] $G$ satisfies $\ac_{pr}$.
    \end{itemize}
\end{theo}
\begin{proof}
By \cite[Proposition 2.1]{Dardo} and \cite[Theorem 2.13]{Dardo2}, the first 6 conditions are equivalent for a radical group. Using, in addition, Proposition \ref{radbyfin} and Proposition \ref{simplepro}, we have that $G$ is radical-by-finite, so that they also hold for any generalized radical group. Each of conditions (i)--(vi) implies (viii) by Lemma \ref{rob}, and Theorem \ref{radual} then yields (vii). Conversely, (vii) clearly implies (vi). Finally, assume that $G$ satisfies (viii). If $G$ is not minimax, the conclusion in (vii) follows again from Theorem \ref{radual}.
\end{proof}

\bigskip

\begin{flushleft}
\rule{8cm}{0.4pt}\\
\end{flushleft}

{
\sloppy
\noindent
Mattia Brescia

\noindent 
Dipartimento di Matematica e Applicazioni ``Renato Caccioppoli''

\noindent
Università di Napoli Federico II

\noindent
Complesso Universitario Monte S. Angelo

\noindent
Via Cintia, Napoli (Italy)

\noindent
e-mail: mattia.brescia@unina.it

}

\bigskip\bigskip

{
\sloppy
\noindent
Bernardo Giuseppe Di Siena, Alessio Russo

\noindent 
Dipartimento di Matematica e Fisica

\noindent
Universit`a degli Studi della Campania ``Luigi Vanvitelli''

\noindent
Viale Lincoln 5, Caserta (Italy)

\noindent
e-mail: bernardogiuseppe.disiena@unicampania.it; alessio.russo@unicampania.it 

}

\end{document}